\documentclass[11pt,english]{article}
\usepackage{amsfonts,amsmath,amsthm,amscd,amssymb,latexsym,cite,verbatim,texdraw,floatflt,
caption2,pb-diagram}

\usepackage{tikz} 
\usetikzlibrary{positioning,trees, decorations.markings}

\theoremstyle{plain}
\newtheorem{theorem}{Theorem}[section]
\newtheorem{lemma}{Lemma}[section]

\newtheorem{corollary}{Corollary}[section]
\newtheorem{definition}{Definition}[section]
\theoremstyle{definition}
\newtheorem{example}{Example}[section]
\newtheorem{remark}{Remark}[section]
\newcommand{\keywords}{\textbf{Key words. }\medskip}
\newcommand{\subjclass}{\textbf{MSC 2020. }\medskip}
\renewcommand{\abstract}{\textbf{Abstract. }\medskip}

\numberwithin{equation}{section}

\newcommand{\RR}{\mathbb{R}}

\begin{document}

\title{A note on generalized four-point inequality}

\author{Evgeniy A. Petrov, Ruslan R. Salimov}

\date{}

\maketitle

\begin{abstract}
In 2017 M. Bessenyei and Z. P\'ales~\cite{BP17} introduced a definition of a triangle function which generates a concept of a generalized triangle inequality in semimetric spaces. Inspired by this concept we discuss already known inequalities in metric spaces that relate the six distances determined by four points  and introduce a definition of a generalized four-point inequality in semimetric spaces. Conditions under which quasisymmetric mappings and quasim\"{o}bius mappings between semimetric spaces preserve such inequality are obtained.
\end{abstract}




\section{Introduction}
The fundamental concept of metric space was introduced by M. Fr\'{e}chet~\cite{Fr06} in 1906. Fr\'{e}chet called the discovered spaces ``classes (D)'' (from the word ``distance''). In 1914 F.~Hausdorff~\cite{Ha14} introduced the term ``metric space''  considering these spaces as a special case of infinite topological spaces. Later there were many attempts to generalize the concept of metric space. One of the directions of such generalization is the generalization of the concept of triangle inequality.  For example, if the usual triangle inequality in the metric space $(X,d)$ is replaced by $d(x,y)\leqslant K(d(x,z)+d(z,y))$, $K\geqslant 1$, then  $(X,d)$ is called a \emph{b-metric space}.
This relaxed triangle inequality and the corresponding spaces were rediscovered several times under various names, see~\cite{CC20} and references therein.
An another example is so-called power triangle inequality considered by D.~J. Greenhoe~\cite{G16}.
The next important for us generalization is a concept of a triangle function~(\ref{e21}) introduced by  M. Bessenyei and Z. P\'ales~\cite{BP17} in 2017, see Definition~\ref{d21} below. This definition allows to study metric spaces from more general positions easily obtaining corollaries for special types of metric spaces by an appropriate choice of the function $\Phi$. On the other hand, in the study of metric spaces, inequalities that relate the six distances defined by four points are very often considered. The most well-known of them are Ptolemy's inequality~(\ref{e270}), additive inequality~(\ref{e61}), inequality which determines $\delta$-hyperbolic metric~(\ref{e62}), etc. The aim of this paper is to introduce a four-point analog of triangle function~(\ref{e21}) and to demonstrate a technique of working with this generalization on the example of problems related to preservation of four-point inequalities by quasisymmetric and quasim\"{o}bius mappings.

Recall the necessary definitions. Let $X$ be a nonempty set. A mapping  $d\colon X\times X\to \mathbb{R}^+$, $\mathbb{R}^+=[0,\infty)$ is a \emph{metric} if for all $x,y,z \in X$ the following axioms hold:
\begin{itemize}
  \item [(i)] $(d(x,y)=0)\Leftrightarrow (x=y)$,
  \item [(ii)] $d(x,y)=d(y,x)$,
  \item [(iii)] $d(x,y)\leqslant d(x,z)+d(z,y)$.
\end{itemize}
The pair $(X,d)$ is called a \emph{metric space}. If only axioms (i) and (ii) hold then $d$ is called a \emph{semimetric}. A pair $(X,d)$, where  $d$  is a semimetric on $X$, is called a \emph{semimetric space}.
Such spaces were first examined by Fr\'{e}chet in~\cite{Fr06}, where he called them ``classes (E)''. Later these spaces attracted the attention of many mathematicians ~\cite{Ch17,Ni27,Wi31,Fr37, BD12}.

Note that in the literature a different terminology is used.
Sometimes a semimetric space is called a distance space~\cite{G16}; a semimetric also can be called a dissimilarity~\cite{DF98}. In~\cite[Chapter 10]{KV88}, the topological space $(X, \tau_d)$ with the topology generated by $d$ is called a symmetric space, whereas a semimetric space means a symmetric space in which all open balls are neighbourhoods.
In our paper we inherit the terminology from Wilson's pioneering paper~\cite{Wi31}, as well as it is adopted in a well-known Blumenthal's work~\cite[p.~7]{Bl53} and many recent papers, e.g.,~\cite{DP13, BP17, DH17, KS15}.

M. Bessenyei and Z. P\'ales~\cite{BP17} in the definition of a triangle function took $\overline{\RR}_+^2$ and $\overline{\RR}^+$ as the domain and the range of $\Phi$, respectively.
We use this definition in a slightly different form restricting the domain and the range of $\Phi$ by ${\RR}_+^2$ and ${\RR}^+$, respectively.

\begin{definition}\label{d21}
Consider a semimetric space $(X, d)$. We say that $\Phi \colon {\RR}^+\times{\RR}^+ \to {\RR}^+$ is a \emph{triangle function} for $d$ if $\Phi$ is symmetric and monotone increasing in both of its arguments, satisfies $\Phi(0,0)=0$ and, for all $x, y, z \in X$, the generalized triangle inequality
\begin{equation}\label{e21}
d(x,y)\leqslant \Phi(d(x,z), d(y,z))
\end{equation}

holds.
\end{definition}

Obviously, metric spaces are semimetric spaces with the triangle function $\Phi(u, v) = u + v$.
In~\cite{BP17} those semimetric spaces whose so-called basic triangle functions are continuous at the origin were considered. These spaces were termed regular. It was shown that the topology of a regular semimetric space is Hausdorff, that a convergent sequence in a regular semimetric space has a unique limit and possesses the Cauchy property, etc. See also~\cite{JT20,VH17,CJT18} for some new results in this direction.

In~\cite{PS22} considering quasisymmetric mappings between semimetric spaces with different triangle functions it was found  a new estimation for the ratio of diameters of two subsets, which are images of two bounded subsets. This result generalizes the well-known Tukia-V\"{a}is\"{a}l\"{a} inequality. Conditions under which quasisymmetric mappings preserve triangle functions and Ptolemy's inequality were also found.

Let us consider several already known different types of inequalities that relate the six distances determined by four points in (semi)metric space.
\subsection{Ptolemy's inequality}
A metric space $(X,d)$ is called Ptolemaic if for all $x, y, z, t \in X$ the inequality
\begin{equation}\label{e270}
d(x,z)d(t,y)\leqslant d(x,y)d(t,z)+d(x,t)d(y,z).
\end{equation}
holds, see, e.g.,~\cite{Sch40, Sch52}. Every pre-Hilbert space is Ptolemaic (see~\cite[9.7.3.8, 10.9.2]{Ber09} for instance), and each linear quasinormed Ptolemaic space is a pre-Hilbert space~\cite{Sch52}.  The Ptolemy theorem, known since antiquity, states that~(\ref{e270}) turns into equality when $x$, $y$, $z$, and $t$ are the vertices of a convex quadrilateral inscribed into a circle.
It is also known~\cite{BFW09} that CAT(0) spaces are Ptolemaic. However a
geodesic Ptolemaic space is not necessarily a CAT(0) space.
On the other hand, a metric space is CAT(0) if and only if it is Ptolemaic and Busemann convex~\cite{FLS07}. Ptolemaic spaces still attract attention of many mathematicians, see, e.g.,~\cite{As18, DP11, MS13}. See also~\cite[Chapter 10]{KS14} for the extended review on Ptolemaic spaces.

In what follows under Ptolemaic spaces we understand semimetric spaces $(X,d)$ for which inequality~(\ref{e270}) holds. Note that~(\ref{e270}) does not imply the standard triangle inequality in $(X,d)$.
\subsection{Additive  inequality}
A metric $d$ on $X$ is \emph{additive}~\cite[p.~7]{DD09} if it satisfies the following strengthened version of the triangle inequality called the four-point inequality:
\begin{equation}\label{e61}
d(x,z) + d(t,y)\leqslant \max\{d(x,y) + d(t,z), d(x,t) + d(y,z)\}
\end{equation}
for all $x$, $y$, $z$, $t\in  X$. Equivalently, among the three sums $d(x,z) + d(t,y)$, $d(x,y) + d(t,z)$, $d(x,t) + d(y,z)$ the two largest sums are equal.
The well-known Buneman's criterion~\cite[Theorem 2]{Bu74} asserts that a finite metric space is additive if and only if it is a tree metric. Note that tree metrics play an important role in phylogenetics~\cite{SS03} and hierarchical clustering~\cite{AC11}.

Recall that an ultrametric is a metric for which the strong triangle inequality $d(x,y)\leqslant \max\{d(x,z),d(z,y)\}$ holds. The class of ultrametric spaces is contained in the class of additive metric spaces.

\subsection{$\delta$-hyperbolic metric}
Given a number $\delta\geqslant 0$, a metric $d$ on a set $X$ is called $\delta$-hyperbolic if it satisfies the Gromov $\delta$-hyperbolic inequality (another weakening of
the four-point inequality):
\begin{equation}\label{e62}
d(x,z) + d(t,y)\leqslant 2\delta+ \max\{d(x,y) + d(t,z), d(x,t) + d(y,z)\}
\end{equation}
for all $x$, $y$, $z$, $t\in  X$, see~\cite[p.~8]{DD09}. A metric space $(X, d)$ is $0$-hyperbolic exactly when $d$ satisfies the four-point inequality~(\ref{e61}). Every bounded metric space of diameter $D$ is $D$-hyperbolic. The $n$-dimensional hyperbolic space is $\ln 3$-hyperbolic. Every $\delta$-hyperbolic metric space is isometrically embeddable into a
geodesic metric space~\cite{BS00}.

\subsection{Roundness of a metric space}
The roundness of a metric space $(X, d)$ is the supremum of all $q$ such
that
\begin{equation}\label{e53}
d(x_1, x_2)^q+d(y_1, y_2)^q \leqslant  d(x_1, y_1)^q+d(x_1, y_2)^q+d(x_2, y_1)^q+d(x_2, y_2)^q
\end{equation}
for any four points $x_1, x_2, y_1, y_2 \in X$, see~\cite[p. 24]{DD09}.
Every metric space has roundness $\geqslant 1$; it is $\leqslant 2$ if the space has approximate midpoints. The roundness of $L_p$-space is $p$ if $1 \leqslant  p \leqslant  2$.

\subsection{Quadrilateral inequality for CAT(0)-spaces}
The Berg and Nikolaev theorem states
that any geodesic space $(X, d)$ is a CAT(0)-space iff for any quadruple of points
$a, b, c, d \in X$ we have
\begin{equation*}
d(a, d)^2 + d(b, c)^2 \leqslant  d(a, b)^2 + d(b, d)^2 + d(d, c)^2 + d(c, a)^2,
\end{equation*}
see~\cite{BN08}.

\subsection{Reshetnyak's quadruple inequality}

The inequality
\begin{equation}\label{e54}
d(x_1, x_3)^2+d(x_2, x_4)^2 \leqslant  d(x_2, x_3)^2+d(x_4, x_1)^2+2d(x_1, x_2)d(x_3, x_4),
\end{equation}
$x_1, x_2, x_3, x_4 \in X$, characterizes Hadamard spaces  among geodesic metric spaces $(X,d)$, see~\cite{BN08}. For the case of Hadamard spaces,~(\ref{e54}) is known as Reshetnyak's quadruple inequality~\cite{S03} or quadrilateral inequality~\cite{BN08} and can be interpreted as a generalization of the Cauchy--Schwartz inequality to metric spaces~\cite{BN08}. Note that inequality~(\ref{e54}) does not hold in every metric space. In~\cite{S19} a condition, which is called quadruple inequality and generalizes~(\ref{e54}), was  introduced and used to establish convergence rates in probability and expectation for spaces with infinite diameter.

The aim of the following definition is to generalize    as many as possible the above considered inequalities.
\begin{definition}\label{d0}
Consider a semimetric space $(X, d)$. Let  $\Phi,\Psi \colon {\RR}^+\times{\RR}^+ \to {\RR}^+$ be  symmetric and monotone increasing in both of their arguments and $\Psi(0,0)=0$.
We shall say that the pair $(\Phi, \Psi)$ forms a \emph{quadruple inequality} in the space $(X,d)$, if for all $x,y,z,t \in X$, the inequality
\begin{equation}\label{e00}
\Psi\big(d(x,z),d(t,y)\big) \leqslant \Phi\Big(\Psi\big(d(x,y),d(t,z)\big), \Psi\big(d(x,t),d(y,z)\big)\Big)
\end{equation}
holds.
\end{definition}

Clearly, if $\Psi(u,v)=uv$ and $\Phi(u,v)=u+v$, then $(X,d)$ is {Ptolemaic space}.
If $\Psi(u,v)=u+v$ and $\Phi(u,v)=2\delta+\max\{u,v\}$, $\delta\geqslant 0$, then $d$ is {$\delta$-hyperbolic}. If additionally $\delta=0$, then $(X,d)$ is {additive metric space}.
If $\Psi(u,v)=u^q+v^q$ and $\Phi(u,v)=u+v$, then we have inequality~(\ref{e53}), which defines the roundness of the metric space $(X,d)$. In the case $q=2$ we obtain the quadrilateral inequality for $CAT(0)$ spaces. It remains to note that~(\ref{e00}) does not generalize Reshetnyak's quadruple inequality, but, if necessary, such generalizations easily can be made by analogy.

We introduce the following two definitions, which are special cases of Definition~\ref{d0}, in order to simplify formulations of the results obtained in the rest of the paper.

\begin{definition}\label{d1}
Consider a semimetric space $(X, d)$. Let  $\Phi \colon {\RR}^+\times{\RR}^+ \to {\RR}^+$ be  symmetric and monotone increasing in both of its arguments and $\Phi(0,0)=0$.
We shall say that $\Phi$ is a \emph{multiplicative quadruple function} for the space $(X,d)$ if for all $x,y,z,t \in X$ the following \emph{multiplicative quadruple inequality}
\begin{equation}\label{e1}
d(x,z)d(t,y) \leqslant \Phi\big(d(x,y)d(t,z), d(x,t)d(y,z)\big)
\end{equation}
holds.
\end{definition}

\begin{definition}\label{d2}
Consider a semimetric space $(X, d)$. Let  $\Phi \colon {\RR}^+\times{\RR}^+ \to {\RR}^+$ be  symmetric and monotone increasing in both of its arguments and $\Phi(0,0)=0$.
We say that $\Phi$ is a \emph{additive quadruple function} for the space $(X,d)$ if for all $x,y,z,t \in X$, the following \emph{additive quadruple inequality}
\begin{equation}\label{e2}
d(x,z)+d(t,y) \leqslant \Phi\big(d(x,y)+d(t,z), d(x,t)+d(y,z)\big)
\end{equation}
holds.
\end{definition}

The following lemma gives us sufficient conditions imposed on functions $\Phi$ and $\Psi$ under which quadruple inequality~(\ref{e00}) holds for the points $x, y, z, t \in Y$ in the case when at least two points among them coincide.

\begin{lemma}\label{l11}
Let $(Y,\rho)$ be a semimetric space, let $\Phi$ and $\Psi$ satisfy conditions of Definition~\ref{d0} and additionally
\begin{equation}\label{le811}
a \leqslant \Phi(0,a)  \, \text{ for all } \, a>0,
\end{equation}
\begin{equation}\label{le812}
\Psi \text{ is a triangle function in } (Y,\rho),
\end{equation}
\begin{equation}\label{le810}
\Psi(0,a) \leqslant \Phi(0,a) \, \text{ for all } \,  a>0.
\end{equation}
Suppose that among the points $x, y, z, t \in Y$ at least two points coincide. Then inequality~(\ref{e00}) holds.
\end{lemma}
\begin{proof}
Without loss of generality consider the following cases:

(i) $z=x$ and  $x$, $y$, $t$ are different. According to~(\ref{e00}) we have to show the inequality
$$
\Psi\big(0,\rho(t,y)\big) \leqslant \Phi\Big(\Psi\big(\rho(x,y),\rho(t,x)\big), \Psi\big(\rho(x,t),\rho(y,x)\big)\Big).
$$
Indeed, using consecutively the monotonicity of $\Phi$, ~(\ref{le812}), ~(\ref{le811}) and ~(\ref{le810}) we have
$$
\Phi\Big(\Psi\big(\rho(x,y),\rho(t,x)\big), \Psi\big(\rho(x,t),\rho(y,x)\big)\Big)\geqslant
\Phi\Big(0, \Psi\big(\rho(x,t),\rho(y,x)\big)\Big)\geqslant
$$
$$
\Phi\Big(0, \rho(t,y)\Big) \geqslant \Psi\big(0,\rho(t,y)\big).
$$

(ii) $y=x$ and  $x$, $t$, $z$ are different.
According to~(\ref{e00}) we have to show the inequality
$$
\Psi\big(\rho(x,z),\rho(t,x)\big) \leqslant \Phi\Big(\Psi\big(0,\rho(t,z)\big), \Psi\big(\rho(x,t),\rho(x,z)\big)\Big),
$$
which, clearly, follows from the monotonicity of $\Phi$,~(\ref{le811}) and symmetry of $\Psi$.

(iii) $z=t=x$ and  $x$, $y$ are different. According to~(\ref{e00}) we have to show the inequality
$$
\Psi\big(0,\rho(x,y)\big) \leqslant \Phi\Big(\Psi\big(\rho(x,y),0\big), \Psi\big(0,\rho(y,x)\big)\Big)
$$
which, clearly, follows from the monotonicity of $\Phi$ and~(\ref{le811}).

(iv) $y=t=x$ and  $x$, $z$ are different.
According to~(\ref{e00}) we have to show the inequality
$$
\Psi\big(\rho(x,z),0\big) \leqslant \Phi\Big(\Psi\big(0,\rho(x,z)\big), \Psi\big(0,\rho(x,z)\big)\Big)
$$
which, clearly, follows from the monotonicity of $\Phi$,~(\ref{le811}) and symmetry of $\Psi$.

(v) $z=x$, $t=y$ and  $x$, $y$ are different.
According to~(\ref{e00}) we have to show the inequality
$$
\Psi\big(0,0\big) \leqslant \Phi\Big(\Psi\big(\rho(x,y),\rho(y,x)\big), \Psi\big(\rho(x,y),\rho(y,x)\big)\Big),
$$
which is evident.

(iv) $y=x$, $z=t$ and  $x$, $z$ are different.
According to~(\ref{e00}) we have to show the inequality
$$
\Psi\big(\rho(x,t),\rho(t,x)\big) \leqslant \Phi\Big(\Psi\big(0,0\big), \Psi\big(\rho(x,t),\rho(x,t)\big)\Big),
$$
which, clearly, follows from~(\ref{le811}).

(v) The case $x=y=z=t$ is evident.
\end{proof}

The following lemma gives us sufficient conditions imposed on function $\Phi$ and semimetric $\rho$ under which multiplicative quadruple inequality~(\ref{e1}) and additive quadruple inequality~(\ref{e2}) hold for the points $x, y, z, t \in Y$ in the case when at least two points among them coincide.

\begin{lemma}\label{l12}
Let $(Y,\rho)$ be a semimetric space and let $\Phi$ satisfy conditions of Definition~\ref{d0}. Suppose that among the points $x, y, z, t \in Y$ at least two points coincide.

(i) If additionally $\Phi$ satisfies condition~(\ref{le811}) then inequality~(\ref{e1}) holds.

(ii) If additionally $\Phi$ satisfies condition~(\ref{le811}) and $(Y,\rho)$ is a metric space then inequality~(\ref{e2}) holds.
\end{lemma}

\begin{proof}
(i) It suffices to set $\Psi(u,v)=uv$ in~Lemma~\ref{l11}  and to note that the case (i) of Lemma~\ref{l11} can be easily shown without conditions~(\ref{le812}) and~(\ref{le810}).

(ii) It suffices to set $\Psi(u,v)=u+v$ in~Lemma~\ref{l11}.
\end{proof}

\section{Quasisymmetric mappings preserving \\ four-point inequalities}

\subsection{Quasisymmetric mappings preserving the generalized four-point inequality}

Quasisymmetric mappings on the real line were first introduced by Beurling and Ahlfors~\cite{BA}.
They found a way to obtain a quasiconformal extension of a quasisymmetric self-mapping of the real axis to a self-mapping of the upper half-plane.
This concept was later generalized by Tukia and V\"{a}is\"{a}l\"{a}~\cite{TV80}, who studied quasisymmetric mappings between general metric spaces.
In the recent years, these mappings are being intensively studied by many mathematicians, see e.g.,~\cite{AKT05, T98, HL15, BM13, YVZ18}
In~\cite{PS22} the definition of quasisymmetric mappings was extended  to the case of general semimetric spaces in the following way.
\begin{definition}\label{d11}
Let $(X,d)$, $(Y,\rho)$ be semimetric spaces. We shall say that a mapping $f\colon X\to Y$ is $\eta$-\emph{quasisymmetric} if there is a homeomorphism $\eta\colon [0, \infty)\to [0,\infty)$ so that
\begin{equation}\label{e0}
d(x,a)\leqslant t d(x,b)  \, \text{ implies } \, \rho(f(x),f(a))\leqslant \eta(t)\rho(f(x),f(b))
\end{equation}
for all triples $a,b,x$ of points in $X$ and for all $t>0$.
\end{definition}

Let $(X,d)$ be a semimetric space. Denote by
$$B(a,r)=\{x\in X\colon d(a,x)<r\}$$
a ball in the space $(X,d)$ with the center $a$ and the radius $r$. Let for any pair of points $a, b \in X$ and any $\varepsilon>0$ there exists $\delta>0$ such that for all $x\in B(a,\delta)$ and $y\in B(b,\delta)$ the inequality $|d(x,y)-d(a,b)|<\varepsilon$ holds. In this case the semimetric $d$ is called continuous. In particular, any metric $d$ has this property. The topology generated by a continuous semimetric is always a Hausdorff topology. The system of sets of the form $B(a,r)$ gives the neighborhood base at the point $a$ in this topology. Note that under a supposition in Definition~\ref{d11} that $d$ and $\rho$ are continuous it follows  that $f$ is either a constant or an embedding, i.e., injective continuous mapping between semimetric spaces with continuous inverse mapping~\cite[Proposition 2.7]{PS22}.

The following theorem gives us the set of necessary conditions under which an $\eta$-{quasisymmetric} mapping $f\colon X\to Y$ preserves a quadruple inequality.

\begin{theorem}\label{lp3.11}
Let $(X,d)$, $(Y,\rho)$ be semimetric spaces. Suppose that the pair $(\Phi,\Psi)$ forms a quadruple inequality in $(X,d)$ and $\Phi$, $\Psi$ 
satisfy conditions~(\ref{le811}), (\ref{le812}), (\ref{le810}) and additionally

\begin{equation}\label{le211}
k\Phi(u,v)= \Phi(ku,kv) \, \text{ for all } \, k>0,
\end{equation}
\begin{equation}\label{le212}
\Psi\left(\frac{a}{b},\frac{c}{d}\right)= \frac{\Psi(a,b)}{\Psi(c,d)} \quad \text{for all } \, a,b,c,d>0.
\end{equation}
Let $f\colon X\to Y$ be a bijective $\eta$-quasisymmetric mapping. If for every $t_1, t_2, t_3, t_4\in \RR^+\setminus\{0\}$ the inequality
\begin{multline}\label{le271}
1\leqslant \Phi\left(\Psi\left(\frac{1}{t_1},\frac{1}{t_2}\right),\Psi\left(\frac{1}{t_3},\frac{1}{t_4}\right)\right)
\, \text{ implies } \, \\
1\leqslant \Phi\left(\Psi\left(\frac{1}{\eta(t_1)},\frac{1}{\eta(t_2)}\right),\Psi\left(\frac{1}{\eta(t_3)},\frac{1}{\eta(t_4)}\right)\right),
\end{multline}
then $(\Phi,\Psi)$ also forms a quadruple inequality in $(Y,\rho)$.
\end{theorem}

\begin{proof}
Let $x, y, z, t \in X$ be different points and let $x'=f(x)$, $y'=f(y)$, $z'=f(z)$, $t'=f(t)$.  By~(\ref{e00}), ~(\ref{le211}) and~(\ref{le212})  we have
\begin{equation*}
1\leqslant \Phi \left(\Psi\left(\frac{d(x,y)}{d(x,z)},\frac{d(t,z)}{d(t,y)}\right), \Psi\left(\frac{d(x,t)}{d(x,z)},\frac{d(y,z)}{d(t,y)}\right)\right).
\end{equation*}
Set
$$
\frac{d(x,z)}{d(x,y)}= t_1, \quad \frac{d(t,y)}{d(t,z)}= t_2, \quad \frac{d(x,z)}{d(x,t)}= t_3, \quad \frac{d(t,y)}{d(y,z)}= t_4.
$$
Hence,
\begin{equation}\label{le291}
1\leqslant \Phi\left(\Psi\left(\frac{1}{t_1},\frac{1}{t_2}\right),\Psi\left(\frac{1}{t_3},\frac{1}{t_4}\right)\right).
\end{equation}

By~(\ref{e0}) we have
$$
\rho(f(x),f(z))\leqslant \eta(t_1) \rho(f(x),f(y)), \quad
\rho(f(t),f(y))\leqslant \eta(t_2) \rho(f(t),f(z)),
$$
$$
\rho(f(x),f(z))\leqslant \eta(t_3) \rho(f(x),f(t)), \quad
\rho(f(t),f(y))\leqslant \eta(t_4) \rho(f(y),f(z)).
$$
Hence, since $\Psi$ is monotone increasing in both of its arguments we have
$$
\Psi\left(\frac{1}{\eta(t_1)},\frac{1}{\eta(t_2)}\right)
\leqslant \Psi\left(\frac{\rho(x',y')}{\rho(x',z')},\frac{\rho(t',z')}{\rho(t',y')}\right)
$$
and
$$
\Psi\left(\frac{1}{\eta(t_3)},\frac{1}{\eta(t_4)}\right)
\leqslant
\Psi\left(\frac{\rho(x',t')}{\rho(x',z')},\frac{\rho(y',z')}{\rho(t',y')}\right).
$$
Since $\Phi$ is monotone increasing we have
$$
\Phi\left(\Psi\left(\frac{1}{\eta(t_1)},\frac{1}{\eta(t_2)}\right),\Psi\left(\frac{1}{\eta(t_3)},\frac{1}{\eta(t_4)}\right)\right) $$
$$
\leqslant
\Phi\left(\Psi\left(\frac{\rho(x',y')}{\rho(x',z')},\frac{\rho(t',z')}{\rho(t',y')}\right), \Psi\left(\frac{\rho(x',t')}{\rho(x',z')},\frac{\rho(y',z')}{\rho(t',y')}\right)\right).
$$
Using consecutively condition~(\ref{le291}), implication~(\ref{le271})
and the previous inequality we get
$$
1\leqslant
\Phi\left(\Psi\left(\frac{\rho(x',y')}{\rho(x',z')},\frac{\rho(t',z')}{\rho(t',y')}\right), \Psi\left(\frac{\rho(x',t')}{\rho(x',z')},\frac{\rho(y',z')}{\rho(t',y')}\right)\right).
$$
By~(\ref{le212}) we have
$$
1\leqslant
\Phi\left(
\frac{\Psi(\rho(x',y'),\rho(t',z'))}{\Psi(\rho(x',z'),\rho(t',y'))}
, \frac{\Psi(\rho(x',t'),\rho(y',z'))}{\Psi(\rho(x',z'),\rho(t',y'))}
\right).
$$
Using~(\ref{le211}) we obtain
$$
1\leqslant
\frac{\Phi\left(
\Psi(\rho(x',y'),\rho(t',z'))
, \Psi(\rho(x',t'),\rho(y',z'))
\right)}{\Psi(\rho(x',z'),\rho(t',y'))}.
$$
Hence, we have inequality~(\ref{e00}) for  $x', y', z', t'$.

The case when among the points $x', y', z', t' \in X$ at least two points coincide is proved in Lemma~\ref{l11},  which completes the proof.
\end{proof}

\begin{remark}
Note that the general solution of the equation $k\Phi(x,y)=\Phi(kx,ky)$, $k\neq 0$, has the form
$\Phi(x,y)=xg(y/x)$, where $g(x)$ is an arbitrary function, see~\cite[p. 429]{PM98}.
\end{remark}

\begin{corollary}\label{p3.11}
Let $(X,d)$, $(Y,\rho)$ be semimetric spaces, $\Phi$ be a multiplicative quadruple function for $(X,d)$ such that condition~(\ref{le811}) holds and
\begin{equation*}
k\Phi(u,v)= \Phi(ku,kv) \, \text{ for all } \,  k>0,
\end{equation*}
and let $f\colon X\to Y$ be a surjective $\eta$-quasisymmetric mapping. If for every $t_1, t_2, t_3, t_4\in \RR^+$ the inequality
\begin{multline*}
t_1t_2t_3t_4 \leqslant \Phi(t_1t_2,t_3t_4)
\, \text{ implies } \, \\
\eta(t_1)\eta(t_2)\eta(t_3)\eta(t_4)\leqslant \Phi(\eta(t_1)\eta(t_2),\eta(t_3)\eta(t_4)),
\end{multline*}
then $\Phi$ is also a multiplicative quadruple function for $(Y,\rho)$.
\end{corollary}
\begin{proof}
It suffices to set $\Psi(u,v)=uv$ in Theorem~\ref{lp3.11} and to apply Lemma~\ref{l12} (i).
\end{proof}

The following corollary was initially obtained in~\cite{PS22} by the direct proof.

\begin{corollary}\label{c3.11}
Let $(X,d)$ be a Ptolemaic space, $(Y,\rho)$ be a semimetric space and let $f\colon X\to Y$ be a bijective $\eta$-quasisymmetric mapping. If for every $t_1, t_2, t_3, t_4\in \RR^+$ the inequality
\begin{multline}\label{e2.8}
t_1t_2t_3t_4 \leqslant t_1t_2+t_3t_4
\, \text{ implies } \, \\
\eta(t_1)\eta(t_2)\eta(t_3)\eta(t_4)\leqslant \eta(t_1)\eta(t_2)+\eta(t_3)\eta(t_4),
\end{multline}
then $(Y,\rho)$ is also Ptolemaic.
\end{corollary}
\begin{proof}
It suffices to set $\Phi(u,v)=u+v$ in Corollary~\ref{p3.11}.
\end{proof}

The following assertion is well-known, see, e.g., Section 2.12 in~\cite{HLP52}.
\begin{lemma}\label{l3.1}
If $0<\alpha\leqslant 1$, then for $u,v\geqslant 0$ the inequality
$(u+v)^{\alpha}\leqslant u^{\alpha}+v^{\alpha}$
holds.
\end{lemma}

\begin{example}\label{e214}
Let $X$ be a Ptolemaic space, $Y$ be a semimetric space and let $f\colon X\to Y$ be a bijective $\eta$-quasisymmetric mapping such that $\eta(t)=t^{\alpha}$, $0<\alpha\leqslant 1$. Then $Y$ is also a Ptolemaic space. Indeed, monotonicity of $\eta$, the first inequality in~(\ref{e2.8}) and Lemma~\ref{l3.1} imply
\begin{equation*}
t_1^{\alpha}t_2^{\alpha}t_3^{\alpha}t_4^{\alpha} \leqslant (t_1t_2+t_3t_4)^{\alpha}\leqslant t_1^{\alpha}t_2^{\alpha}+t_3^{\alpha}t_4^{\alpha}
\end{equation*}
which by Corollary~\ref{c3.11} completes the proof.
\end{example}

\subsection{Quasisymmetric mappings preserving additive four-point inequality}

The following theorem gives us the set of necessary conditions under which an $\eta$-{quasisymmetric} mapping $f\colon X\to Y$ preserves an additive quadruple inequality.

\begin{theorem}\label{t4.5}
Let $(X,d)$ be a semimetric space, $(Y,\rho)$ be a metric space and let $\Phi$ be an additive quadruple function for $(X,d)$ satisfying~(\ref{le811}) and such that
\begin{equation}\label{le511}
k\Phi(u,v)= \Phi(ku,kv) \,  \text{ for all } \, k>0,
\end{equation}
and let $f\colon X\to Y$ be a bijective $\eta$-quasisymmetric  mapping. If for every $t_1, t_2, t_3, t_4, t_5\in \RR^+\setminus\{0\}$ the inequality
\begin{equation}\label{e51}
1+\frac{1}{t_1}\frac{1}{t_5}\leqslant \Phi\left(\frac{1}{t_1}+\frac{1}{t_2}, \frac{1}{t_3}+\frac{1}{t_4}\right)
\end{equation}
 implies
\begin{equation}\label{e52}
1+\eta\left(\frac{1}{t_1}\right)\eta\left(\frac{1}{t_5}\right)\leqslant \Phi\left(\frac{1}{\eta(t_1)}+\frac{1}{\eta(t_2)}, \frac{1}{\eta(t_3)}+\frac{1}{\eta(t_4)}\right),
\end{equation}
then $\Phi$ is also an additive quadruple function for $(Y,\rho)$.
\end{theorem}

\begin{proof}
Let $x, y, z, t \in X$ be different points and let $x'=f(x)$, $y'=f(y)$, $z'=f(z)$, $t'=f(t)$.  Dividing both parts of~(\ref{e2}) on $d(x,z)$,
using~(\ref{le511}) and the evident equality
$$
\frac{d(t,y)}{d(x,z)}=
\frac{d(x,y)}{d(x,z)}
\frac{d(t,y)}{d(x,y)}
$$
we have
\begin{equation}\label{e281}
1 \leqslant \Phi\left( \frac{d(x,y)}{d(x,z)} + \frac{d(t,z)}{d(x,z)}, \frac{d(x,t)}{d(x,z)} + \frac{d(y,z)}{d(x,z)}\right) -
\frac{d(x,y)}{d(x,z)}
\frac{d(t,y)}{d(x,y)}.
\end{equation}
Set
$$
\frac{d(x,z)}{d(x,y)}= t_1, \quad \frac{d(x,z)}{d(t,z)}= t_2, \quad
\frac{d(x,z)}{d(x,t)}= t_3, \quad \frac{d(x,z)}{d(y,z)}= t_4, \quad \frac{d(x,y)}{d(t,y)}= t_5.
$$
Clearly, ~(\ref{e281}) implies~(\ref{e51}).

By~(\ref{e0}) we have
$$
\rho(f(x),f(z))\leqslant \eta(t_1) \rho(f(x),f(y)), \quad
\rho(f(x),f(z))\leqslant \eta(t_2) \rho(f(t),f(z)),
$$

$$
\rho(f(x),f(z))\leqslant \eta(t_3) \rho(f(x),f(t)), \quad
\rho(f(x),f(z))\leqslant \eta(t_4) \rho(f(y),f(z)),
$$

$$
\rho(f(x),f(y))\leqslant \eta\left(\frac{1}{t_1}\right) \rho(f(x),f(z)).
$$
$$
\rho(f(t),f(y))\leqslant \eta\left(\frac{1}{t_5}\right) \rho(f(x),f(y))
$$

Hence, by the monotonicity of $\Phi$ we obtain
$$
\Phi\left(\frac{1}{\eta(t_1)} + \frac{1}{\eta(t_2)}, \frac{1}{\eta(t_3)} + \frac{1}{\eta(t_4)}\right) - \eta\left(\frac{1}{t_1}\right)\eta\left(\frac{1}{t_5}\right)
$$
$$
\leqslant \Phi\left( \frac{\rho(x',y')}{\rho(x',z')} + \frac{\rho(t',z')}{\rho(x',z')}, \frac{\rho(x',t')}{\rho(x',z')} + \frac{\rho(y',z')}{\rho(x',z')}\right) -
\frac{\rho(x',y')}{\rho(x',z')}
\frac{\rho(t',y')}{\rho(x',y')}.
$$
Using the last inequality,~(\ref{e52}) and~(\ref{le511}) we obtain additive quadruple inequality~(\ref{e2}) in the space $(Y,\rho)$ for the points $x', y', z', t'$.

The case when among the points $x', y', z', t'$ at least two points coincide follows from Lemma~\ref{l12} (ii), which completes the proof.
\end{proof}

The following corollary was initially obtained in~\cite[Theorem 4.1]{PS21B} by the direct proof.

\begin{corollary}\label{t4.1}
Let $(X,d)$ be an additive metric space, $(Y,\rho)$ be a metric space and let $f\colon X\to Y$ be a bijective $\eta$-quasisymmetric mapping. If for every $t_1, t_2, t_3, t_4, t_5\in \RR^+\setminus\{0\}$ the inequality
\begin{equation*}
1+\frac{1}{t_1}\frac{1}{t_5}\leqslant \max\left\{\frac{1}{t_1}+\frac{1}{t_2}, \frac{1}{t_3}+\frac{1}{t_4}\right\}
\end{equation*}
 implies
\begin{equation*}
1+\eta\left(\frac{1}{t_1}\right)\eta\left(\frac{1}{t_5}\right)\leqslant \max\left\{\frac{1}{\eta(t_1)}+\frac{1}{\eta(t_2)}, \frac{1}{\eta(t_3)}+\frac{1}{\eta(t_4)}\right\},
\end{equation*}
then $(Y,\rho)$ is also an additive metric space.
\end{corollary}
\begin{proof}
It suffices to set $\Phi(u,v)=\max\{u,v\}$ in Theorem~\ref{t4.5}.
\end{proof}

\begin{example}\label{e4.2}
If $\eta(t)=t$, then an $\eta$-quasisymmetric mapping preserves additive metrics.
\end{example}

\section{Quasim\"{o}bius mappings preserving the multiplicative quadruple inequality}

It is well-known that cross-ratios are invariant under M\"{o}bius transformations.
As a generalization of M\"{o}bius transformations in metric spaces V\"{a}is\"{a}l\"{a}~\cite{Va84}
introduced a class of quasim\"{o}bius maps, under which the cross
ratio is in a certain sense quasi-invariant. The introduction of quasim\"{o}bius maps
has provided a handy tool when studying the quasisymmetric maps and the quasiconformal
maps, see~\cite{WZ17} and references therein.

The tetrad in a semimetric space $(X,d)$ is a quadruple $T=(x,y,z,t)$ of mutually distinct points. The absolute cross-ratio of a tetrad $T$ is
$$
R(T)=R(x,y,z,t)=\frac{d(x,y)d(t,z)}{d(x,z)d(t,y)}.
$$

\begin{definition}
Let $(X,d)$, $(Y,\rho)$ be semimetric spaces. We shall say that a mapping $f\colon X\to Y$ is $\eta$-\emph{quasim\"{o}bius} if it is bijective and there is a homeomorphism $\eta\colon [0, \infty)\to [0,\infty)$ so that the estimate
 \begin{equation}\label{e43}
   R(fT)\leqslant \eta(R(T))
 \end{equation}
holds for any $T=(x,y,z,t)$ in $X$, where $fT=(f(x),f(y),f(z),f(t))$.
\end{definition}

In particular, $f$ is a \emph{m\"{o}bius} mapping if $R(fT)=R(T)$ for each tetrad $T$ in $X$.

\begin{theorem}\label{t3.5}
Let $(X,d)$, $(Y,\rho)$ be semimetric spaces, $\Phi$ be a multiplicative quadruple function for $(X,d)$ satisfying~(\ref{le811}) and such that
\begin{equation}\label{le411}
k\Phi(u,v)= \Phi(ku,kv) \, \text{ for all } \, k>0,
\end{equation}
and let $f\colon X\to Y$ be a bijective $\eta$-quasim\"{o}bius  mapping. If for every $t_1, t_2\in \RR^+\setminus\{0\}$ the inequality
\begin{equation}\label{le471}
1\leqslant \Phi\left(t_1, t_2\right)
\, \text{ implies } \,
1\leqslant \Phi\left(\frac{1}{\eta\left({t_1}^{-1}\right)},
\frac{1}{\eta\left({t_2}^{-1}\right)}\right),
\end{equation}
then $\Phi$ is  also a multiplicative quadruple function for $(Y,\rho)$.
\end{theorem}

\begin{proof}
Let $T=(x,y,z,t)$ be a tetrad in $X$. Then since $f$ is a bijection  $fT=(f(x),f(y),f(z),f(t))= (x',y',z',t')$ is a tetrad in $Y$.

By~(\ref{e1}) and~(\ref{le411}) we have
\begin{equation*}
1\leqslant \Phi\left(\frac{d(x,y)}{d(x,z)}\frac{d(t,z)}{d(t,y)},\frac{d(x,t)}{d(x,z)}\frac{d(y,z)}{d(t,y)}\right).
\end{equation*}
Set
$$
t_1 = \frac{d(x,y)}{d(x,z)}\frac{d(t,z)}{d(t,y)}, \quad
t_2 = \frac{d(x,t)}{d(x,z)}\frac{d(y,z)}{d(t,y)}.
$$
Hence,
\begin{equation}\label{e491}
1\leqslant \Phi({t_1},{t_2}).
\end{equation}

By~(\ref{e43}) we have
$$
\frac{\rho(x',z')}{\rho(x',y')}\frac{\rho(t',y')}{\rho(t',z')}\leqslant \eta\left({t_1}^{-1}\right), \quad
\frac{\rho(x',z')}{\rho(x',t')}\frac{\rho(t',y')}{\rho(y',z')} \leqslant \eta\left({t_2}^{-1}\right).
$$
Hence,
$$
\Phi\left(\frac{1}{\eta\left({t_1}^{-1}\right)},
\frac{1}{\eta\left({t_2}^{-1}\right)}\right)
\leqslant \Phi\left(\frac{\rho(x',y')}{\rho(x',z')}\frac{\rho(t',z')}{\rho(t',y')},
\frac{\rho(x',t')}{\rho(x',z')}\frac{\rho(y',z')}{\rho(t',y')}\right).
$$
Using~(\ref{e491}), implication~(\ref{le471}) and~(\ref{le411}) we obtain inequality~(\ref{e1}) for  $x', y', z', t'$.

The case when among the points $x', y', z', t'$ at least two points coincide follows from Lemma~\ref{l12} (i),  which completes the proof.
\end{proof}

\begin{corollary}\label{p27}
Let $(X,d)$ be a Ptolemaic space, $(Y,\rho)$ be a semimetric space and let $f\colon X\to Y$ be a bijective $\eta$-quasim\"{o}bius  mapping. If for every $t_1, t_2\in \RR^+\setminus \{0\}$ the inequality
\begin{equation}\label{e24}
1 \leqslant t_1+t_2
\, \text{ implies } \,
1\leqslant \frac{1}{\eta\left({t_1}^{-1}\right)}+
\frac{1}{\eta\left({t_2}^{-1}\right)},
\end{equation}
then $(Y,\rho)$ is also Ptolemaic.
\end{corollary}
\begin{proof}
It suffices to set $\Phi(u,v)=u+v$ in Theorem~\ref{t3.5}.
\end{proof}

\begin{example}\label{c214}
Let $X$ be a Ptolemaic space, $Y$ be a semimetric space and let $f\colon X\to Y$ be a bijective $\eta$-quasisym\"{o}bius mapping such that $\eta(t)=t^{\alpha}$, $0<\alpha\leqslant 1$. Then $Y$ is also a Ptolemaic space. Indeed, the first inequality in~(\ref{e24}), the monotonicity of $\eta$ and Lemma~\ref{l3.1} imply
\begin{equation*}
1\leqslant (t_1+t_2)^{\alpha}\leqslant t_1^{\alpha}+t_2^{\alpha}=\frac{1}{\eta\left({t_1}^{-1}\right)}+
\frac{1}{\eta\left({t_2}^{-1}\right)},
\end{equation*}
which by Corollary~\ref{p27} completes the proof.
\end{example}

\bigskip

CONTACT INFORMATION

\medskip

Evgeniy Aleksandrovych Petrov \\
Institute of Applied Mathematics and Mechanics of the NAS of Ukraine, Slavyansk, Ukraine \\
E-Mail: eugeniy.petrov@gmail.com

\medskip

Ruslan Radikovych Salimov\\
Institute of Mathematics of the NAS of Ukraine, Kyiv, Ukraine\\
E-Mail: ruslan.salimov1@gmail.com
\end{document}